\documentclass[draft,reqno%
]{amsart}

\usepackage{amsmath,amsfonts,amsthm}
\usepackage[french,english]{babel}

\makeatletter

\renewcommand
\DeclareRobustCommand\citey{%
  \@ifnextchar [{\@tempswatrue\@citey}{\@tempswafalse\@citey[]}}
\renewcommand\@citex [2][]{%
  \let\@citea\@empty
  \@cite{\@for\@citeb:=#2\do
    {\@citea\def\@citea{,\penalty\@m\ }%
     \edef\@citeb{\expandafter\@firstofone\@citeb\@empty}%
     \if@filesw\immediate\write\@auxout{\string\citation{\@citeb}}\fi
     \@ifundefined{b@\@citeb}{\mbox{\reset@font\bfseries ?}%
       \G@refundefinedtrue
       \@latex@warning
         {Citation `\@citeb' on page \thepage \space undefined}}%
       {%\hbox{                    !!!!!!!!!!! Changing LaTeX
       \csname b@\@citeb\endcsname%}!! (to allow line-breaks)!!!
       }}}{#1}}
\newcommand\@citey [2][]{%
  \let\@citea\@empty
  \@cite{\@for\@citeb:=#2\do
    {\@citea\def\@citea{,\penalty\@m\ }%
     \edef\@citeb{\expandafter\@firstofone\@citeb\@empty}%
     \if@filesw\immediate\write\@auxout{\string\citation{\@citeb}}\fi
     \@ifundefined{b@\@citeb}{\mbox{\reset@font\bfseries ?}%
       \G@refundefinedtrue
       \@latex@warning
         {Citation `\@citeb' on page \thepage \space undefined}}%
       {%\hbox{                    !!!!!!!!!!! Changing LaTeX
       \csname b@year\@citeb\endcsname%}!! (to allow line-breaks,
       }}}{#1}}%  !!!!!!! and to print only the year, not the author)!!!
\renewcommand\@newl@bel[3]{%
  \@ifundefined{#1@#2}%
    \relax
    {\gdef \@multiplelabels {%
       \@latex@warning@no@line{There were multiply-defined labels}}%
     \@latex@warning@no@line{Label `#2' multiply defined}}%
  \global\@namedef{#1@#2}{#3}%
  \global\@namedef{#1@year#2}{\@secondoftwo #3}}% !! Printing only the
\makeatother

\newtheorem{theorem}{Theorem}[section]
\newtheorem{lemma}[theorem]{Lemma}

\newtheorem{corollary}[theorem]{Corollary}
\newtheorem{proposition}[theorem]{Proposition}

\theoremstyle{definition}
\newtheorem{definition}[theorem]{Definition}
\newtheorem{example}[theorem]{Example}
\newtheorem{notation}[theorem]{Notation}

\theoremstyle{remark}
\newtheorem{remark}[theorem]{Remark}

\numberwithin{equation}{theorem}
\makeatletter
\@addtoreset{equation}{section}
\makeatother

\makeatletter
\thm@notefont{\mdseries\upshape}  % !!!!!!!!!!!!!! Changing amsthm.sty!
\def\thmhead@plain#1#2#3{%              !!!!!!!!!! Changing amsthm.sty!
  \thmname{#1}\thmnumber{\@ifnotempty{#1}{ }%\@upn{Changing amsthm.sty!
  #2%}                              !!!!!!!!!!!!!! Changing amsthm.sty!
}%                                  !!!!!!!!!!!!!! Changing amsthm.sty!
  \thmnote{ {\the\thm@notefont(#3)}}}
\let\thmhead\thmhead@plain
\makeatother

\title[Two-Variable Pierce-Birkhoff for generalized polynomials]
{Extension of the Two-Variable Pierce-Birkhoff\\
conjecture to generalized polynomials}

\thanks
{To appear in {\it Annales de la Facult\'e des
Sciences de Toulouse\/}. %\hfil\break
The results in this paper were first presented at the
Conference on Ordered Rings (``Ord007"), at Louisiana
State University, Baton Rouge, Louisiana, USA, April 25--28,
2007: %\hfil\break
{\tt http://www.math.lsu.edu/%\textasciitilde!!!!!!!!!!!!!!!!!!!!!
$\sim$madden/Ord007}.}

\dedicatory{In honor of Melvin Henriksen's 80th birthday}

\author[Charles N.~Delzell]{Charles N.~Delzell}

\address
{Department of Mathematics\\
Louisiana State University\\
Baton Rouge, Lou\-i\-si\-ana 70803\\
USA}

\email{delzell@math.lsu.edu}

\urladdr{www.math.lsu.edu/%\textasciitilde!!!!!!!!!!!!!!!!!!!!!
$\sim$delzell}

\date{November 21, 2009}



\begin{document}

\begin{abstract}
[English version:]
Let $h:\mathbb R^n\to\mathbb R$ be a continuous, piecewise-polynomial
function. The Pierce-Birkhoff conjecture (1956) is that any such $h$ is
representable in the form $\sup_i\inf_jf_{ij}$, for some finite
collection of polynomials $f_{ij}\in\mathbb R[x_1,\ldots,x_n]$.
(A simple example is $h(x_1)=|x_1|=\sup\{x_1,-x_1\}$.)
In 1984, L.~Mah\'e and, independently, G.~Efroymson,
proved this for $n\le2$; it remains open for $n\ge3$. In this paper
we prove an analogous result for ``generalized polynomials'' (also
known as signomials), i.e., where the exponents are allowed to be
arbitrary real numbers, and not just natural numbers; in this version,
we restrict to the positive orthant, where each $x_i>0$. As before,
our methods work only for $n\le2$.

\smallskip
[French version:]
\selectlanguage{french}
En 1984, L.~Mah\'e, et ind\'ependammant G.~Efroymson, ont prouv\'e
le cas o\`u $n\le2$ de la conjecture de Pierce-Birkhoff (1956) :
une fonction $h:\mathbb R^n\to\mathbb R$ continue
po\-ly\-no\-mi\-ale par morceaux peut s'\'ecrire comme
$\sup_i\inf_jf_{ij}$, pour une collection finie de
polyn\^omes $f_{ij}\in\mathbb R[x_1,\ldots,x_n]$.
(Un exemple simple est $h(x_1)=|x_1|=\sup\{x_1,-x_1\}$.)
La conjecture reste ouverte pour $n\ge3$.
Dans cet article, nous prouvons (encore pour $n\le2$)
un r\'esultat analogue pour \og polyn\^omes
g\'en\'eralis\'es\fg, o\`u les exposants peuvent
\^etre des nombres r\'eels arbitraires, et non pas seulement
des nombres naturels; dans cette version, nous
limitons le domaine \`a l'orthant positif, o\`u chaque $x_i>0$.

\selectlanguage{english}

\iffalse
Below is a non-TeX version of the abstract, in pure ASCII text
(i.e., with no TeX macros, and no dollar signs):

Let h: R^n --> R be a continuous, piecewise-polynomial
function (here, R denotes the real numbers).
The Pierce-Birkhoff conjecture (1956) is that
any such h is representable in the form sup_i inf_j f_ij,
for some finite collection of polynomials f_ij in R[x_1,...,x_n].
(A simple example, for n=1, is h(x) = |x| = sup{x,-x}.)
In 1984, L. Mahe and, independently, G. Efroymson,
proved this for n < 3; it remains open for n > 2. In this paper
we prove an analogous result for "generalized polynomials"
(also known as signomials), i.e., where the exponents are allowed to be
arbitrary real numbers, and not just natural numbers;
in this version, we restrict to the positive orthant, where each x_i > 0.
As before, our methods work only for n < 3.
\fi
\end{abstract}

\maketitle

\section{Generalized polynomial functions\\
and generalized semialgebraic sets}
\label{signomials}

We write $\mathbb R_+=[0,\infty)$ and $\mathbb R_{++}=(0,\infty)$,
endowed with the usual, order topology. And the Cartesian product,
$\mathbb R_{++}^2:=\mathbb R_{++}\times\mathbb R_{++}$ will
be endowed with the usual, Euclidean topology.

\begin{definition}
\label{gpf}
A {\em generalized polynomial function\/} $a(x,y)$ of two
variables is a function $a:\mathbb R_{++}^2\to\mathbb R$
of the form
\begin{equation}
\label{signomial}
a:=a(x,y):=
c_1x^{\alpha_{1,1}}y^{\alpha_{1,2}}+
c_2x^{\alpha_{2,1}}y^{\alpha_{2,2}}+\dotsb+
c_mx^{\alpha_{m,1}}y^{\alpha_{m,2}},
\end{equation}
where $m\in\mathbb N:=\{0,1,2,\ldots\}$, the ``coefficients"
$c_i$ of $a$ are nonzero elements of $\mathbb R$,
and the (binary) ``exponents" $\alpha_i:=
(\alpha_{i,1},\alpha_{i,2})$
of $a$ are distinct elements of $\mathbb R^2$.
We write $\mathbb R[\mathbb R^2]$ for the ring (actually,
it is a group ring) of all generalized polynomial
functions $a:\mathbb R_{++}^2\to\mathbb R$.
\end{definition}

Thus, generalized polynomial functions (sometimes called
``signomial'' functions) of two variables can be defined, roughly,
as ``real polynomial functions on $\mathbb R_{++}^2$
with arbitrary real exponents.'' A simple example is
$a(x,y)=y-x^\pi$.

Generalized polynomial functions of two variables are
clearly real analytic on $\mathbb R_{++}^2$.

See
%{\tracingall
\cite{Delzell 2008} %}
for background on the general properties
and the history of generalized polynomials (in any number
of variables), and some motivation for studying them.

\begin{definition}
\label{sss}
We call a subset $A\subseteq\mathbb R_{++}^2$
a {\it generalized semialgebraic set\/}, or a {\it
semisignomial set\/}, if it is of the form
$\bigcup_{j=1}^JS_j$, where $J\in\mathbb N$ and each $S_j$
is a ``basic semisignomial'' set, i.e., one of the form
\begin{equation}
S_j=\{\,(x,y)\in\mathbb R_{++}^2\mid
f_j(x,y)=0,\ g_{j,1}(x,y)>0,\ldots,g_{j,K_j}(x,y)>0\,\},
\label{gss}
\end{equation}
where each $K_j\in\mathbb N$ and the $f_j$ and $g_{jk}$
are generalized polynomials.
\end{definition}

(Recall that ordinary semialgebraic subsets of $\mathbb R^2$
or $\mathbb R^n$ are defined analogously, but with the
$f_j$ and $g_{jk}$ being (ordinary) polynomials.)

\section{Piecewise generalized polynomial functions}

\begin{definition}
\label{pgp}
We call a function $h(x,y):\mathbb R_{++}^2\to\mathbb R$ a
{\it piecewise generalized polynomial function\/} of two
variables if there exist $g_1,\ldots,g_l\in\mathbb R
[\mathbb R^2]$ \eqref{gpf} such that the subsets
\begin{equation}
\label{A_i}
A_i:=\{\,(x,y)\in\mathbb R_{++}^2\mid h(x,y)=g_i(x,y)\,\}
\end{equation}
are generalized semialgebraic and cover $\mathbb R_{++}^2$, i.e.,
$\mathbb R_{++}^2=\bigcup_iA_i$.

We may, and shall, assume that the $g_i$ are distinct.
\end{definition}

\begin{example}
\label{example of h}
~\nopagebreak

\begin{picture}(110,110)(-180,-10)
\put(-180,50)
{$\displaystyle
h(x,y):=\begin{cases}
y-x^\pi&\hbox{if }y\ge x^\pi,\\
0      &\hbox{if }y< x^\pi.
\end{cases}$}

{\thicklines
\put(-5,0){\vector(1,0){105}}
\put(0,-5){\vector(0,1){105}}}

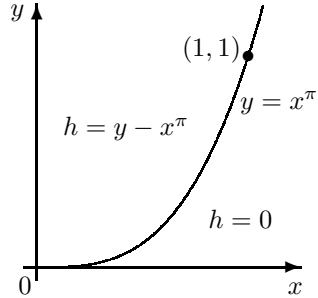
\put(-7,-10)0
\put(-10,95){$y$}
\put(95,-10){$x$}
\qbezier( 0, 0    )(13.633, 0    )(20  , 1.027)
\qbezier(20, 1.027)(31.265, 2.845)(40  , 9.065)
\qbezier(40, 9.065)(50.758,16.724)(60  ,32.403)
\qbezier(60,32.403)(70.543,50.290)(80  ,80    )
\qbezier(80,80    )(82.836,88.910)(85.6,98.947)
\put(80,80){\circle*4}
\put(55,80){$(1,1)$}
\put(77,60){$y=x^\pi$}
\put(10,50){$h=y-x^\pi$}
\put(65,15){$h=0$}
\end{picture}
\end{example}

The following, technical lemma will not be needed until
Proposition~\ref{one variable} and Lemma~\ref{finer partition}
below, and can be skipped on a first reading.
In it, for any set $A$ in $\mathbb R_{++}^2$, we shall write $A^\circ$
for the interior of $A$.

\begin{lemma}
\label{union}
Let $A_1,\ldots,A_l$ be as in \eqref{pgp}.

$(1)$ \vrule width0pt depth15pt
$\displaystyle\bigcup_{i=1}^lA_i^\circ$
is dense in $\mathbb R_{++}^2$.

$(2)$ \vrule width0pt depth5pt
$A_i^\circ\cap A_j^\circ=\emptyset$ for $i\ne j$.

$(3)$ If $h$ is continuous, then each $A_i$ is closed, whence
$\overline{A_i^\circ}\subseteq A_i$.

$(4)$ If $h$ is continuous, then
$\displaystyle
\bigcup_{i=1}^lA_i^\circ=\mathbb R_{++}^2\setminus
\!\!\bigcup_{1\le i<j\le l}\!\!\bigl(\overline{A_i^\circ}
\cap\overline{A_j^\circ}\bigr)$.

$(5)$ Suppose $h$ is continuous, and $E$ is a connected subset
of $\mathbb R_{++}^2$ such that for each $(x,y)\in E$, the $l$
values $g_1(x,y),g_2(x,y),\ldots,g_l(x,y)$ are distinct.
Then there exists an $i\in\{1,2,\ldots,l\}$
such that $E\subseteq A_i^\circ$ $($in particular, such that $h=g_i$
throughout $E)$. This $i$ is unique in case $E\ne\emptyset$.

\end{lemma}

\begin{proof}
(1) By \eqref{sss}, $\bigcup_iA_i$ is a combined, but still finite,
union of suitable basic semisignomial sets $S_j$ as in \eqref{gss}.
Let $T$ be the union of those $S_j$ for which $f_j\not\equiv0$;
thus, $T\subseteq Z(F):=\{\,(x,y)\in\mathbb R_{++}^2\mid F(x,y)=0\,\}$,
where $F$ is the product of those $f_j$'s.
$\mathbb R_{++}^2\setminus Z(F)$ is dense in $\mathbb R_{++}^2$,
by the identity theorem for real analytic functions.
{\it A fortiori\/}, $\mathbb R_{++}^2\setminus T$ is also
dense in $\mathbb R_{++}^2$. The union $U$ of the other $S_j$'s
(viz., those for which $f_j\equiv0$) must contain
$\mathbb R_{++}^2\setminus T$ (since $T\cup U=\bigcup_iA_i=\mathbb R_{++}^2$
\eqref{pgp}), and so $U$ is also dense in $\mathbb R_{++}^2$.
But $\bigcup_iA_i^\circ\supseteq U$.\footnote
{In fact, $\bigcup_iA_i^\circ=U$. But we don't need this.}

(2) If $A_i^\circ\cap A_j^\circ\ne\emptyset$, then $g_i$ would
agree with $g_j$ on a nonempty open set (by \eqref{A_i}),
and hence on all of $\mathbb R_{++}^2$ (again by the identity
theorem), contradicting the distinctness
of the $g_i$ in \eqref{pgp}.\footnote
{And if $g_i$ agrees with $g_j$ on all of $\mathbb R_{++}^2$,
then the coefficients of $g_i$ and $g_j$ (i.e., the $c$'s in
\eqref{signomial} above) would agree, too, by
\cite[Remark~4.3]{Delzell 2008}.}

(3) Obvious.

(4) $\subseteq$. Let $(x,y)\in A_i^\circ$ and suppose $j\ne i$.
It is enough to show that $(x,y)\notin\overline{A_j^\circ}$.
There exists an open disk in $A_i$ about $(x,y)$. In fact,
this disk is in $A_i^\circ$, and hence is disjoint from
$A_j^\circ$, by (2) above. Therefore
$(x,y)\notin\overline{A_j^\circ}$.\thinspace\footnote
{This half of the proof of (4) does not require the hypothesis
that $h$ be continuous.}

$\supseteq$. Suppose $(x,y)\in\mathbb R_{++}^2\setminus
\bigcup_iA_i^\circ$. For $r\in\mathbb R_{++}$ with
$r\le\min\{x,y\}$, let $B_r$ denote the open disk in
$\mathbb R_{++}^2$ of radius $r>0$ about $(x,y)$, and let
$I(r)=\{\,i\in\{1,2,\ldots,l\}\mid B_r\cap A_i^\circ\ne\emptyset\,\}$.
Then for every $r$, $|I(r)|\ge1$, by (1) above.
In fact, $I(r)>1$. Otherwise, for some $i$, $A_i^\circ\cap B_r$
would be dense in $B_r$ (by (1) again), whence $B_r=
\overline{A_i^\circ}\cap B_r\subseteq A_i\cap B_r$ (by (3)),\footnote
{In fact, this inclusion is actually an equality.}
i.e., $B_r\subseteq A_i$, whence $(x,y)\in A_i^\circ$,
contradiction. Now, for any $s\in\mathbb R_{++}$ with $s<r$,
$I(s)\subseteq I(r)$; i.e., the finite set $I(r)$ decreases
monotonically with $r$, and yet always has cardinality $\ge2$.
Thus, there exist at least two indices $i<j$ such that for
every $r\in(0,\min\{x,y\})$, $B_r$ meets $A_i^\circ$ and
$A_j^\circ$. Therefore $(x,y)\in\overline{A_i^\circ}\cap\overline
{A_j^\circ}$.

(5) The distinctness hypothesis of (5) can be rephrased as
\begin{equation*}
E\cap\bigcup_{i<j}(A_i\cap A_j)=\emptyset.
\end{equation*}
{\it A fortiori\/},
$E\cap\bigcup_{i<j}\bigl(\overline{A_i^\circ}\cap\overline{A_j^\circ})
=\emptyset$, using (3). By (4), $E\subseteq\bigcup_iA_i^\circ$.
The existence of the desired $i$ now follows from (2) and the
hypotheses that $E$ is connected. The uniqueness of $i$ in case
$E\ne\emptyset$ also follows from (2).
\end{proof}

\begin{remark}
\label{announcement of Remark "aside"}
In Remark~\ref{aside} below, we shall use \eqref{union} above
to see that when a piecewise
generalized polynomial function $h$ is continuous, each $A_i$
in \eqref{pgp} can automatically be taken to be a generalized
semialgebraic set; it is not necessary to include that
condition as a hypothesis in \eqref{pgp}.
\end{remark}

The set of piecewise generalized polynomial functions is closed
under differences and products, and so forms a ring;
it is also closed under pointwise suprema and infima,
and so forms an $l$-ring under those lattice operations.
(This ring is, of course, even an $f$-ring.) The continuous
functions in this $f$-ring comprise a sub-$f$-ring.
(See, e.g., \cite{Birkhoff et al. 1956} or
\cite{Henriksen et al. 1962} for background
on $l$-rings and $f$-rings.)

\section{Statement and discussion of the main result}

\begin{theorem}[Main Theorem: The Pierce-Birkhoff conjecture for generalized
polynomials in two variables]
\label{PBCSIG}
If $h:\mathbb R_{++}^2\to\mathbb R$ is continuous and piecewise
generalized polynomial, then $h$ is a $($pointwise\/$)$ sup of
infs of finitely many generalized polynomial functions; i.e.,
\begin{equation}
\label{SIP}
h(x,y)=\sup_j\inf_kf_{jk}(x,y)\text{ on }\mathbb R_{++}^2,
\end{equation}
for some finite number of generalized polynomials
$f_{jk}$. \rm(The converse is easy.)
\end{theorem}

\begin{example}
For the $h$ in Example~\ref{example of h} above,
$h(x,y)=\sup\{0,\,y-x^\pi\}$.
\end{example}

The representation of $h$ in the form \eqref{SIP}
makes both the continuity and the piecewise generalized
polynomial character of $h$ {\it obvious\/}.

For ordinary polynomials in $\mathbb R[X,Y]$
and ordinary piecewise polynomial functions on $\mathbb R^2$,
the analog of Theorem~\ref{PBCSIG} above was first proved by
L.~Mah\'e \citey{Mahe 1984} and Efroymson (unpublished),
independently. The statement and proofs of the
Mah\'e-Efroymson theorem generalize easily to the situation
where $\mathbb R$ is replaced by an arbitrary real closed
field $R$ (furnished with the topology induced by the unique
ordering on $R$). But the fact that then the coefficients
of the $f_{jk}$ in the Mah\'e-Efroymson theorem may be
taken to lie in the subfield of $R$ generated by the
coefficients of the $g_i$ defining $h$ (in the analog
of \eqref{pgp}), was not trivial, and was proved in
\cite{Delzell 1989}.

The extension of the Mah\'e-Efroymson
theorem to functions of three or more variables (like the
extension of \eqref{PBCSIG} above) remains unproved
and unrefuted; it is known as the Pierce-Birkhoff Conjecture
(first formulated in \cite{Birkhoff et al. 1956}).

In our proof of Theorem~\ref{PBCSIG} below, we shall make
no attempt to indicate which steps generalize easily to the
case where $n>2$ (though many of those steps do).
The first reason for this is that the
notation is often simpler when $n=2$. The second reason is that,
considering the many mathematicians who have tried to prove the
Pierce-Birkhoff Conjecture for $n>2$, we now lean toward the
opinion that it and Theorem~\ref{PBCSIG} are false for $n>2$.

In 1987 we proved that for all $n\ge1$ and every real
closed field $R$, if $h:R^n\to R$ is ``piecewise-rational''
(i.e., if there are rational functions $g_1,\ldots,g_l\in R(X)$
such that the sets $A_i:=\{\,x\in R^n\mid g_i(x)
\text{ is defined and }h(x)=g_i(x)\,\}$
are s.a.\ and cover $R^n$), then there are finitely many
$f_{jk}\in R(X)$ and there is a $k\in R[X_1,\ldots,X_n]
\setminus\{0\}$ such that for all $x\in R^n$ where
$k(x)\ne0$ (i.e., for ``almost all'' $x\in R^n$),
each $f_{jk}(x)$ is defined and $h(x)=\sup_j\inf_kf_{jk}(x)$;
this is true even if $h$ is not continuous.
This result was announced in
\cite[p.~659]{Delzell 1989}, and proved in \citey{Delzell 1990}.
Madden gave an ``abstract'' version of this result
that applies to arbitrary fields (and not just
$R(X)$); see \cite{Madden 1989}. In \citey{Delzell 2005}
we proved an analog of our 1987 result, for
``generalized piecewise-rational functions''
(i.e., functions that are, piecewise,
quotients of generalized polynomial functions).

The rest of this paper will be devoted to the proof
of Theorem~\ref{PBCSIG}. In \S4 we shall develop the
necessary one-variable machinery; in \S5 we shall
deal with the additional difficulties arising in the
two-variable situation.

\section{One-variable methods}

We imitate Mah\'e's proof as much as possible.

We are given a continuous function
\begin{equation}
h(x,y)=
\begin{cases}
g_1(x,y)&\text{if \ }(x,y)\in A_1\\
\kern14pt\vdots&\kern25pt\vdots\\
g_l(x,y)&\text{if \ }(x,y)\in A_l,
\end{cases}
\label{h}
\end{equation}
where, as in \eqref{pgp}, the $g_i$ are generalized polynomials
and the $A_i$ cover $\mathbb R_{++}^2$.
(Recall from Remark~\ref{announcement of Remark "aside"}
above that the $A_i$ are also, automatically,
generalized semialgebraic; but we don't use this.)
As before, we assume the $g_i$ are distinct.

Write each $a(x,y)\in\mathbb R[\mathbb R^2]\setminus\{0\}$
\eqref{gpf} in the form
\begin{equation}
a_1(x)y^{\beta_1}+a_2(x)y^{\beta_2}+\dotsb+
a_K(x)y^{\beta_K},
\label{K}
\end{equation}
where $K\ge1$, \ $\beta_1<\dotsb<\beta_K\in\mathbb R$,
and each $a_i$ is a nonzero generalized polynomial in $x$.
This representation is unique.

Let ${\mathcal A}=\{\,g_i-g_j\mid1\le i<j\le l\,\}$.
Let $\mathcal B$ be the smallest subset of $\mathbb R[\mathbb R^2]$
containing $\mathcal A$ and closed under the following two
operations, for each $a(x,y)\in\mathcal B$
for which $K>1$ in \eqref{K}:
\begin{align}
a\vrule width0pt depth20pt % !!!! extra skip after this line.!!!
&\mapsto
\begin{cases}
\displaystyle a':=\frac{\partial a}{\partial y}
&\text{if }\beta_1=0\text{, and}\\
\displaystyle y^{-\beta_1}a(x,y)&\text{if }\beta_1\ne0\
\footnotemark;\text{ and}
\end{cases}
\label{closed under partial}\\
a&\mapsto
\begin{cases}
\displaystyle r:=r_a(x,y)=a(x,y)-\frac y{\beta_K}\cdot a'(x,y)
&\text{if }\beta_1=0,\footnotemark\text{ and}\\
a&\text{if }\beta_1\ne0.
\end{cases}
\label{r}
\end{align}
\addtocounter{footnote}{-1}
\footnotetext{This trick (of dividing by $y^{\beta_1}$)
was first used by Sturm
%{\tracingall
\citey{Sturm 1829}%}
.}
\stepcounter{footnote}
\footnotetext{Here we use $\beta_K\ne0$, which follows from
$\beta_1=0$ and $K>1$.}

\begin{remark}
\label{no x involved}
Suppose no $g_i$ involves the variable $x$;
i.e., each $g_i$ is a function of $y$ alone, and is constant in $x$.
Then the same is, of course, true for each $a\in\mathcal A$;
in fact, the same is true even for each $a\in\mathcal B$, in view
of \eqref{closed under partial} and \eqref{r}.
\end{remark}

\begin{lemma}
\label{K-1 terms}
For each $a\in\mathcal B$ for which $K>1$ and $\beta_1=0$,
$a'(x,y)$ and $r_a$ each have exactly $K-1$ $y$-terms.
Consequently, $\mathcal B$ is finite.
\end{lemma}

\begin{proof}
This is clear for $a'(x,y)$.
For $r_a$, observe (a)~that the $K^{\rm th}$ $y$-term
$a_K(x,y)y^{\beta_K}$ in $a$ \eqref{K} is cancelled
out by the $y$-term
\begin{equation*}
\frac{y}{\beta_K}
\bigl(\beta_K\,a_K(x,y)\,y^{\beta_K-1}\bigr)
\end{equation*}
in
\begin{equation}
\label{partial}
\frac{y}{\beta_K}\cdot a'(x,y),
\end{equation}
and
(b)~that the other $y$-terms of
\eqref{partial} involve the $y$-exponents
$\beta_1,\ldots,\beta_{k-1}$, but with coefficients
different from those of the corresponding $y$-terms
of $a$ (since for each $i<K$, $\beta_i/\beta_K\ne1$).
\end{proof}

\begin{lemma}
\label{projection}
There exist $L\in\mathbb N$ and
$\gamma_1<\gamma_2<\dotsb<\gamma_L\in\mathbb R_{++}$ such that,
writing $\gamma_0=0$ and $\gamma_{L+1}=\infty$,
for each $a\in\mathcal B$ and for each $p\in\{0,1,\ldots,L\}$,
the zeros of $a(x,y)$ in the $p$th vertical half
strip $H_p:=(\gamma_p,\gamma_{p+1})\times\mathbb R_{++}$
are the graphs of
continuous, monotonic\footnote{We do not need the monotonicity
of the $\xi_{a,p,j}$ in this paper.\label{monotonic}} ``generalized
semialgebraic''\footnote
{\label{generalized semialgebraic function}We say that
a function is {\it generalized semialgebraic\/}
if its graph, in the product space, is a generalized semialgebraic
set.} functions $y=\xi_{a,p,j}(x)$, $j=1,2,\ldots,s$
$($where $s:=s(a,p)$ satisfies $0\le s\le K\>$\footnote
{Here, $K$ is as in \eqref{K};
in fact, $s$ is even bounded by the number of {\it
alternations in sign\/} in the sequence $a_0(x),\ldots,a_K  (x)$,
by Sturm's generalization \citey{Sturm 1829},
to one-variable generalized polynomials, of the Fourier-Budan theorem
(which contains Descartes' rule of signs as a special case).}$)$
with
\begin{equation*}
(0\mathrel{<)}\xi_{a,p,1}<\cdots<\xi_{a,p,s}
\text{ on }(\gamma_p,\gamma_{p+1}).
\end{equation*}
Moreover,
$\forall a_1,a_2\in\mathcal B$, \ $\forall p\le L$, \
$\forall j_1\le s(a_1,p)$, \
$\forall j_2\le s(a_2,p)$, throughout $(\gamma_p,\gamma_{p+1})
\subseteq\mathbb R_{++}$, only one of the following
three relations holds:
\begin{align}
\begin{split}
\xi_{a_1,p,j_1}&<\xi_{a_2,p,j_2},\\
\xi_{a_1,p,j_1}&=\xi_{a_2,p,j_2}\text{, or}\\
\xi_{a_1,p,j_1}&>\xi_{a_2,p,j_2}.
\label{uniform trichotomy}
\end{split}
\end{align}
\end{lemma}

% !!!!!! Somehow the align environment (above)
% !!!!!! erased the contents of box0 (defined in the preamble).
% !!!!!! So I re-define box0 here:
%\setbox0=\vbox
%{\offinterlineskip\openup.5\jot
%\halign{\hfil$\scriptstyle#$\hfil\cr
%\omega\cr
%\smile\cr}}

Lemma~\ref{projection} and its Corollary~\ref{corollary}
are illustrated in Figure~\ref{cylindrical figure},
which also shows the stack of open connected sets $D_{2,1},
D_{2,2},D_{2,3}$ whose union is a dense open subset of $H_2$
(looking ahead to \eqref{corollary} below).

\bigskip
\begin{figure}[htb]
\setlength\parindent{0pt}
\begin{picture}(327,200)(-143.5,-20)%!!! This for normal size.
%\begin{picture}(327,190)(-200,-10)% !!! This for big size.
{\thicklines
\put(-110, 0){\vector(1,0){280}}
\put(-100,-15){\vector(0,1){200}}
%\put( 130,80){\vector(0,-1){30}}
}
\put(-108,-10)0
\put(-113,175){$y$}
\put(130,-12){$x$}
%\put(133,65){$\proj_2$}
\qbezier(-30,100)(-30,140)( 40,140)
\qbezier( 40,140)(110,140)(110,100)
\qbezier(-30,100)(-30, 60)( 40, 60)
\qbezier( 40, 60)(110, 60)(110,100)
\qbezier(20,100)(70,100)(120,160)
\qbezier(20,100)(70,100)(120, 40)
\put(-60, -5){\line(0,1){10}}
\put(-60, 15){\line(0,1){10}}
\put(-60, 35){\line(0,1){10}}
\put(-60, 55){\line(0,1){10}}
\put(-60, 75){\line(0,1){10}}
\put(-60, 95){\line(0,1){10}}
\put(-60,115){\line(0,1){10}}
\put(-60,135){\line(0,1){10}}
\put(-60,155){\line(0,1){10}}
\put(-30, -5){\line(0,1){10}}
\put(-30, 15){\line(0,1){10}}
\put(-30, 35){\line(0,1){10}}
\put(-30, 55){\line(0,1){10}}
\put(-30, 75){\line(0,1){10}}
\put(-30, 95){\line(0,1){10}}
\put(-30,115){\line(0,1){10}}
\put(-30,135){\line(0,1){10}}
\put(-30,155){\line(0,1){10}}
\put( 20, -5){\line(0,1){10}}
\put( 20, 15){\line(0,1){10}}
\put( 20, 35){\line(0,1){10}}
\put( 20, 55){\line(0,1){10}}
\put( 20, 75){\line(0,1){10}}
\put( 20, 95){\line(0,1){10}}
\put( 20,115){\line(0,1){10}}
\put( 20,135){\line(0,1){10}}
\put( 20,155){\line(0,1){10}}
\put(91.5, -5){\line(0,1){10}}
\put(91.5, 15){\line(0,1){10}}
\put(91.5, 35){\line(0,1){10}}
\put(91.5, 55){\line(0,1){10}}
\put(91.5, 75){\line(0,1){10}}
\put(91.5, 95){\line(0,1){10}}
\put(91.5,115){\line(0,1){10}}
\put(91.5,135){\line(0,1){10}}
\put(91.5,155){\line(0,1){10}}
\put(110, -5){\line(0,1){10}}
\put(110, 15){\line(0,1){10}}
\put(110, 35){\line(0,1){10}}
\put(110, 55){\line(0,1){10}}
\put(110, 75){\line(0,1){10}}
\put(110.5, 95){\line(0,1){10}}
\put(110,115){\line(0,1){10}}
\put(110,135){\line(0,1){10}}
\put(110,155){\line(0,1){10}}
\put(-60,31){\circle*{3}}
\put(-36, 7){$a(x,y)=0$}
\put(-39,10){\line(-1,0){10}}
\put(-49,10){\vector(-1,2){9}}
\put(-28,135){$\xi_{b,2,2}$}
\put(- 5,125){($=\xi^{2,2}$)}
\put(-28, 64){$\xi_{b,2,1}$}
\put(- 5, 68){($=\xi^{2,1}$)}
\put( 25,144){$\xi_{b,3,2}$ ($=\xi^{3,4}$)}
\put( 38,111){$\xi_{c,3,2}$}
\put( 67,105){($=\xi^{3,3}$)}
\put( 38, 88){$\xi_{c,3,1}$}
\put( 67, 88){($=\xi^{3,2}$)}
\put( 25, 52){$\xi_{b,3,1}$ ($=\xi^{3,1}$)}
\put( 88,149){$\xi_{c,4,2}$}
\put(125,131){$\xi_{b,4,2}$}
\put(117,120){($=\xi^{4,3}$)}
\put(122,133){\line(-1,0){10}}
\put(112,133){\vector(-1,-1){10}}
\put(125, 63){$\xi_{b,4,1}$}
\put(117, 52){($=\xi^{4,2}$)}
\put(122, 66){\line(-1,0){10}}
\put(112, 66){\vector(-1,1){10}}
\put( 90, 47){$\xi_{c,4,1}$}
\put(120,150){$\xi_{c,5,2}$ ($=\xi^{5,2}$)}
\put(115, 33){$\xi_{c,5,1}$ ($=\xi^{5,1}$)}
\put(-57,172){$s(1)$}
\put( 90,171){$s(4)$}
\put(-97,170){$s(0)=0$}
\put(-45,165){$=0$}
\put(-20,170){$s(2)=3$}
\put( 40,170){$s(3)=5$}
\put( 94.5,163){$=5$}
\put(125,170){$s(5)=3$}
\put(-83,17){$H_0$}
\put(-51,17){$H_1$}
\put(-11,17){$H_2$}
\put( 48,17){$H_3$}
\put( 95,17){$H_4$}
\put(130,17){$H_5$}
\put(-63,-12){$\gamma_1$}
\put(-33,-12){$\gamma_2$}
\put( 16,-12){$\gamma_3$}
\put( 88,-12){$\gamma_4$}
\put(106,-12){$\gamma_5$}
\put(-16, 37){$D_{2,1}$}
\put(-16, 97){$D_{2,2}$}
\put(-16,155){$D_{2,3}$}
\end{picture}
\caption{Illustrating Lemma~\ref{projection} and Corollary~\ref
{corollary} by showing the zeros in $\mathbb R_{++}^2$
of $a,b,c\in\mathcal B$: the isolated zero
of $a(x,y)$, and the graphs of $y=\xi_{b,p,j}(x)$
and $y=\xi_{c,p,j}(x)$ (which are also the graphs of
$y=\xi^{p,k}(x)$, for suitable $k$).
Here, $L=5$ (the number of $\gamma$'s).}
\label{cylindrical figure}
\end{figure}
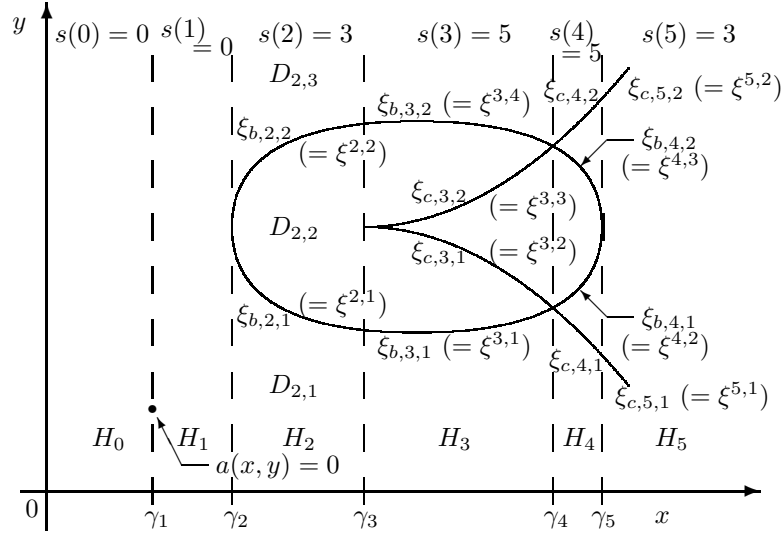

\begin{proof}
Miller \citey{Miller 1994} considered a class of
functions $f:\mathbb R^n\to\mathbb R$ that properly contains
the class of (extensions by 0 to $\mathbb R^n$ of)
generalized polynomial functions. Specifically, he
considered terms built up (in a formal language) from
variable symbols $x_1,x_2,\dotsc$ and from constants in
$\mathbb R$ by the usual operation symbols $+$, $-$, and
$\cdot$~, together with the class of operation symbols
$\{\,x_i^r\mid i\ge1,\ r\in\mathbb R\,\}$;
the symbol $x_i^r$ indicates the function
$\mathbb R\to\mathbb R$ defined by
\begin{equation*}
x_i\mapsto
\begin{cases}
x_i^r&\text{if }x_i>0\\
 0 &\text{if }x_i\le0.
\end{cases}
\end{equation*}
He considered the structure
\begin{equation*}
\mathbb R_{\text{an}}^{\mathbb R}:=
\bigl(\mathbb R,<,+,-,\,\cdot\,,0,1,(x_i^r)_{r\in\mathbb R,\>i\ge1},\,
\bigl(\tilde f\bigr)_{f\in\mathbb R\{X,n\},n\in\mathbb N}\bigr),
\end{equation*}
where $\bigl(\tilde f\bigr)_{f\in\mathbb R\{X,n\},n\in\mathbb N}$
denotes a certain class of functions $\tilde f:
\mathbb R^n\to\mathbb R$
that are analytic on $[-1,1]^n$. He proved that
the theory of $\mathbb R_{\text{an}}^{\mathbb R}$ admits
quantifier-elimination and analytic cell-decomposition,
and is universally axiomatizable, o-minimal, and
polynomially bounded.

The standard properties of o-minimal theories (cf., e.g.,
\cite{Dries 1998} or \cite{Miller 1994}) imply
that the zeros in $\mathbb R_{++}^2$ of all the various
$a\in\mathcal B$ consist of finitely many
isolated points together with the graphs of finitely
many continuous, monotonic functions $\xi_{a,p,j}:
(\gamma_p,\gamma_{p+1})\to\mathbb R_{++}$ (on suitable
intervals $(\gamma_p,\gamma_{p+1})\subseteq\mathbb R_{++}$)
satisfying \eqref{uniform trichotomy}, as stated in the lemma.
(That the $\xi_{a,p,j}$ are generalized semialgebraic is
just the definition of that term (footnote~\ref{generalized
semialgebraic function} above), since
the $a(x,y)$ are generalized polynomials.)
\end{proof}

\begin{notation}
It will be helpful in \eqref{increasing order} below
if we agree that $\xi_{a,p,0}(x)=0$ and
$\xi_{a,p,s+1}(x)=+\infty$ for all $x\in(\gamma_p,\gamma_{p+1})$,
where $p\in\{0,1,\ldots,L\}$ and $s=s(a,p)$ is as in
Lemma~\ref{projection}.
\label{xi_0}
\end{notation}

\begin{corollary}
Let $L$, $\gamma_0,\ldots,\gamma_{L+1}$, and $H_p$
be as in \eqref{projection}% and \eqref{xi_0}
, for some fixed $p\in\{0,1,\ldots,L\}$.
Then the zeros in $H_p$ of all the $a\in\mathcal B$ are
the graphs of continuous, monotonic,
generalized semialgebraic functions $y=\xi^{p,k}(x)$,
$k=1,2,\ldots,s(p)$, where $s(p)$ satisfies
$0\le s(p)\le\sum_{a\in{\mathcal B}}s(a,p)$ $($where $s(a,p)$
is as in \eqref{projection}$)$, and where,
for each $x\in(\gamma_p,\gamma_{p+1})$,
\begin{equation}
0=:\xi^{p,0}(x)<\xi^{p,1}(x)<\cdots<\xi^{p,s(p)}(x)<
\xi^{p,s(p)+1}(x):=\infty.
\label{increasing order}
\end{equation}
Consequently, the sets
\begin{equation*}
D_{p,k}:=\{\,(x,y)\mid
\gamma_p<x<\gamma_{p+1},\>\xi^{p,k}(x)<y<\xi^{p,k+1}(x)\,\},
\end{equation*}
for $k\in\{0,1,\ldots,s(p)\}$, are nonempty, pairwise-disjoint,
generalized semialgebraic cells
$($in particular, they are open and $($pathwise$)$ connected$)$,
and their union is a dense open subset of  $H_p$. Moreover,
the $D_{p,k}$ are ``stacked'' one upon the other in the $y$-direction,
so that for any $x\in(\gamma_p,\gamma_{p+1})$ and for any
$(s(p)+1)$-tuple $y_0,y_1,\ldots,y_{s(p)}\in\mathbb R_{++}$
for which each $(x,y_k)\in D_{p,k}$, $y_0<y_1<\cdots<y_{s(p)}$.
\label{corollary}
\end{corollary}

\begin{proof}
The required sequence $\xi^{p,1},\xi^{p,2},\ldots,\xi^{p,s(p)}$ of
functions is just a suitable permutation and relabelling of the set of
functions $\{\,\xi_{a,p,j}\mid a\in{\mathcal B},\>1\le j\le s(a,p)\,\}$.
That a permutation of the $\xi$'s satisfying \eqref{increasing order}
exists follows from \eqref{uniform trichotomy}.
\end{proof}

\begin{proposition}
\label{closed}
The set of suprema of infima of finitely many generalized polynomial
functions is closed under subtraction and multiplication,
and so is a ring.
\end{proposition}

\begin{proof}
This is a special case of a result of
Henriksen and Isbell \citey[Corollary~3.4]{Henriksen et al. 1962}:
If $S$ is a ring of real-valued functions on a set,
then the least lattice of functions that contains $S$
is also a ring. Here we may take $S=\mathbb R[\mathbb R^2]$
\eqref{gpf}.
For the proof of this corollary,
Henriksen and Isbell gave some $f$-ring identities which,
they said, reduce the proof to an exercise; they omitted
the details. \cite{Delzell 1989} gave a sketch of a proof.
The first complete proof of this fact to appear in print was
that of \cite[Theorem~1(B)]{Hager et al. 2010}; their proof incorporates
some simplifications due to Madden, and their statement is a little
more general than the Henriksen-Isbell statement above,
in that now $S$ may be an arbitrary subring of an arbitrary
$f$-ring.
\end{proof}

In the next lemma it will helpful to use the abbreviation
$a^+=\sup\{0,a\}$, for any real-valued function $a$.

\begin{lemma}[Generalized Mah\'e lemma]
\label{Mahe's lemma}
Using the notation of Lemma~\ref{projection} above,
for each $p\in\{0,1,\ldots,L\}$, each $a(x,y)\in\mathcal B$,
and each $j\in\{0,1,\ldots,s\}$ $($where $s=s(a,p)$
as in \eqref{projection}$)$,
there exists a function $c_{a,p,j}(x,y)$
that is a sup of infs of finitely many generalized
polynomials, such that for all $x\in(\gamma_p,\gamma_{p+1})$
and for all $y\in\mathbb R_{++}$,
\begin{equation}
\label{c}
c_{a,p,j}(x,y)=
\begin{cases}
a(x,y)&\text{if }y>\xi_{a,p,j}(x)\text{, and}\\
0   &\text{otherwise.}
\end{cases}
\end{equation}
\end{lemma}

\begin{proof}
Fix any $p\le L$.

We use induction on $K\ge1$,
the number of distinct $y$-exponents occurring in $a$
(recall \eqref{K}). Note that for any $K\ge1$,
we may (in fact, we must) take $c_{a,p,0}=a$;
this handles the case $K=1$,
i.e., the case where $a(x,y)$ is of the form
$a_1(x)y^{\beta_1}$ (which implies $s(a,p)=0$ for each
$p\le L$).

Now assume $K>1$.

We claim that we may assume
\begin{equation}
\label{beta_1=0}
\beta_1=0.
\end{equation}
If not, then write $b(x,y)=y^{-\beta_1}a(x,y)$.
Thus $b\in\mathcal B$, by \eqref{closed under partial}.
Note that $b(x,y)$
has the same positive $y$-roots $\xi$ as $a(x,y)$ has;
thus $s(a,p)=s(b,p)$.
Therefore, if for each $j\le s(b,p)$ we can construct
$c_{b,p,j}$ such that
\begin{equation*}
c_{b,p,j}(x,y)=
\begin{cases}
b(x,y)&\text{if }y>\xi_{b,p,j}(x)\text{, and}\\
0   &\text{otherwise,}
\end{cases}
\end{equation*}
then we may, for each $j\le s(a,p)$ ($=s(b,p)$),
take $c_{a,p,j}(x,y)=y^{\beta_1}c_{b,p,j}(x,y)$; the latter
product is a sup of infs of finitely many generalized
polynomials, since $c_{b,p,j}$ is, and since $y^{\beta_1}>0$
for all $y>0$ (or use \eqref{closed}).

Next, recall that $a'$ \eqref{closed under partial} and
$r_a$ \eqref{r} each have exactly $K-1$ $y$-terms,
by \eqref{K-1 terms} and \eqref{beta_1=0}.
Thus we assume, by the inductive hypothesis,
that for every $k\le s(a',p)$ and $l\le s(r_a,p)$,
we can construct $c_{a',p,k}$ and $c_{r_a,p,l}$
satisfying the appropriate analogs of \eqref{c}.
Note that $c_{a',p,k}$ and $c_{r_a,p,l}$ are, in
particular, continuous (either by their form as in
\eqref{c}, or by the fact that they are sups of
infs of finitely many generalized polynomial functions).

Finally, in order to construct $c_{a,p,j}$, we now use
induction on $j\in\{0,1,2,\ldots,\linebreak[0]
s(a,p)\}$.
We have already constructed $c_{a,p,0}$,
so now we assume that $j\in\{1,2,\ldots,\linebreak[0]
s(a,p)\}$ and that $c_{a,p,j-1}$ has
already been constructed with the properties stated in
Lemma~\ref{Mahe's lemma}.

Throughout the rest of this proof, $x$ will range over
$(\gamma_p,\gamma_{p+1})$.
By the uniform trichotomy in \eqref{uniform trichotomy},
all order relations involving the various $\xi$'s
below will hold uniformly for such $x$;
thus we %can
usually write, e.g., $\xi_{a,p,j}$
instead of $\xi_{a,p,j}(x)$.\linebreak
Let $k$ be the smallest index such
that $\xi_{a,p,j}\le\xi_{a',p,k}$ (then $1\le k\le
1+s(a',p)$).\linebreak
Let $l$ be the smallest index such
that $\xi_{a',p,k}\le\xi_{r_a,p,l}$ (then $1\le l\le
1+s(r_a,p)$).
Then
\begin{equation}
\label{Rolle}
\xi_{a',p,k}<\xi_{a,p,j+1}\text{\quad
(unless $\xi_{a',p,k}=\infty$), by Rolle's theorem, and}
\end{equation}
\begin{align}
g(x,y):&=\frac y{\beta_K}c_{a',p,k}(x,y)+c_{r_a,p,l}(x,y)\notag\\
&=
\begin{cases}
\ 0&\text{if }0<y<\xi_{a',p,k},\\
\displaystyle\frac y{\beta_K}a'(x,y)=a(x,y)-r_a(x,y)
&\text{if }\xi_{a',p,k}<y<\xi_{r_a,p,l},\\
\vrule width0pt height13pt
\displaystyle\frac y{\beta_K}a'(x,y)+r_a(x,y)
=a(x,y)&\text{if }\xi_{r_a,p,l}<y,
\end{cases}
\label{g}
\end{align}
where \eqref{g} follows from \eqref{r} and from the
definitions of $c_{a',p,k}$ and $c_{r_a,p,l}$.\footnote
{In \eqref{g}, the inequalities in the case-distinctions
$y<\xi_{a',p,k}$, $\xi_{a',p,k}<y<\xi_{r_a,p,l}$, and
$\xi_{r_a,p,l}<y$ are all strict (i.e., they are all $<$,
and not $\le$).
This strictness is necessary because $\xi_{a',p,k}$
and/or $\xi_{r_a,p,l}$ could be $\infty$.
If either or both of the $\xi$'s are finite,
the corresponding inequalities could be relaxed to
nonstrict inequalities (with $\le$). But even without
such a relaxation, \eqref{g} still uniquely determines $g$
even when $y$ is $\xi_{a',p,k}$ or $\xi_{r_a,p,l}$, since
$g$ is continuous for all $y>0$.}
This function $g$ is a supremum of infima of finitely many
generalized polynomial functions, by \eqref{closed}.

If $a'(x,\xi_{a,p,j})=0$, then
\begin{alignat*}2
\xi_{a',p,k}&=\xi_{a,p,j}&&
\text{by the minimality of $k$, and}\\
\xi_{r_a,p,l}&=\xi_{a',p,k}\quad&&
\text{by \eqref{r} and the minimality of $l$.}
\end{alignat*}
Thus we may take $c_{a,p,j}=g$, by \eqref{g}.

Now suppose, on the other hand, that
\begin{equation}
\label{a' ne 0}
a'(x,\xi_{a,p,j})\ne0
\end{equation}
(recall \eqref{uniform trichotomy}). (Then
\begin{equation}
\xi_{a,p,j}<\xi_{a',p,k}.)
\label{xi inequality}
\end{equation}
We may assume that in fact
\begin{equation}
\label{a'>0}
a'(x,\xi_{a,p,j})>0,
\end{equation}
by \eqref{uniform trichotomy}, by replacing $a$
with $-a$, and by the fact that $-c_{-a,p,j}$
($=c_{a,p,j}$) will still be a supremum of infima
of finitely many generalized polynomial functions
if $c_{-a,p,j}$ is, by \eqref{closed}.
Then
\begin{alignat}2
\label{a<0}
a(x,y)&<0\quad&&\text{for}\quad\xi_{a,p,j-1}<y<a_{a,p,j}
\quad\text{and}\\
\label{a>0}
a(x,y)&>0\quad&&\text{for}\quad\xi_{a,p,j}<y<a_{a,p,j+1},
\end{alignat}
by \eqref{a'>0}.

First suppose $\xi_{a',p,k}=\infty$ (i.e.,
$k=1+s(a',p)$). Then $a'(x,y)>0$ for all $y>\xi_{a,p,j}$,
whence $a(x,y)>0$ for all $y>\xi_{a,p,j}$.
Hence we may take $c_{a,p,j}=\inf\{c_{a,p,j-1}^+,a^+\}$,
using also \eqref{a<0}.

Second, suppose $\xi_{a',p,k}<\infty$ (i.e.,
$k\le s(a',p)$). Then
\begin{align}%{alignat}2
r_a(x,\xi_{a',p,k})
&=a(x,\xi_{a',p,k})-
\frac{\xi_{a',p,k}}{\beta_K}a'(x,\xi_{a',p,k})\ \ %&&
\text{(by \eqref{r})}\notag\\
&=a(x,\xi_{a',p,k})-\frac{\xi_{a',p,k}}{\beta_K}\cdot0
\notag\\
&=a(x,\xi_{a',p,k})>0,\quad %&&
\text{by \eqref{a>0},
       \eqref{Rolle}, and \eqref{xi inequality}}.
\label{r>0}
\end{align}%{alignat}
Then for $\xi_{a',p,k}\le y<\xi_{r_a,p,l}$:
\begin{alignat}2
r_a(x,y)&>0&&\text{by \eqref{r>0} and the choice of $l$,
and}\label{r(x,y)>0}\\
g(x,y)&=a(x,y)-r_a(x,y)\quad&&\text{by \eqref{g}}\notag\\
                    &<a(x,y)&&\text{by \eqref{r(x,y)>0}.}
\label{1}
\end{alignat}
Then
\begin{equation*}
\sup\{a,g\}=
\left\lbrace
\begin{alignedat}3
&a^+&\quad&\text{if }0<y\le\xi_{a,p,j}&\ &\text{ by \eqref{g}, and}\\
&a  &     &\text{if }y\ge\xi_{a,p,j}  &&
\text{ by \eqref{g}, \eqref{1}, \eqref{Rolle},
and \eqref{a>0}.}
\end{alignedat}
\right.
\end{equation*}
Therefore, we may take $c_{a,p,j}=
\inf\{c_{a,p,j-1}^+,\,\sup\{a,g\}\}$, by \eqref{a<0}.
\end{proof}

\begin{proposition}
\label{one variable}
Let $h$, $\mathcal A$, and $\mathcal B$ be as before Lemma~\ref{K-1 terms},
and let $L$ and $H_p$ be as in Lemma~\ref{projection}, for some fixed
$p\in\{0,1,\ldots,L\}$. Then there is a function $d_p:\mathbb
R_{++}^2\to\mathbb R$ that $(1)$~is a supremum of infima of
finitely many generalized polynomial functions $\in\mathbb
R[\mathbb R^2]$ and $(2)$~coincides
with $h(x,y)$ on $H_p$.
\end{proposition}

\begin{proof}
Let $\gamma_p$ and $\gamma_{p+1}$ be as in Lemma~\ref{projection},
and let $s(p)$, $\xi^{p,0},\ldots,\xi^{p,s(p)+1}$,
and $D_{p,0},\ldots,D_{p,s(p)}$ be as in Corollary~\ref{corollary}.

For each $k=0,1,\ldots,s(p)$ there exists a unique
$\mu:=\mu(p,k)\in\{1,2,\ldots,l\}$ such that
$D_{p,k}\subseteq A_\mu$ (hence $h=g_\mu$ on $D_{p,k}$,
by \eqref{h}), using Lemma~\ref{union}(5) and the fact
that each $g_i-g_j$ is nonzero throughout $D_{p,k}$.

If $s(p)=0$, we may define the required $d_p$ to be
$g_{\mu(p,0)}\in\mathbb R[\mathbb R^2]$. If $s(p)>0$, then we
shall define $d_p$ as follows. For $k=0,1,\ldots,s(p)-1$, let
$v_{p,k}:=g_{\mu(p,k+1)}-g_{\mu(p,k)}$. We have $v_{p,k}=0$ on
$\overline{D_{p,k}}\cap\overline{D_{p,k+1}}$, since $h$ is
continuous. We extend the notation $c_{a,p,j}$ of
Lemma~\ref{Mahe's lemma} from the case where $a\in\mathcal B$ to the
case where $a=0$: for $j=0,1,\ldots$, we define the function
$c_{0,p,j}$ by $c_{0,p,j}(x,y)=0$ $\forall(x,y)\in\mathbb R_{++}^2$.
If $v_{p,k}\ne0$, then $v_{p,k}\in{\mathcal A}\subset\mathcal B$, so by
\eqref{projection} and \eqref{corollary} there exists
a unique $j(p,k)\in\{1,2,\ldots,s(v_k,p)\}$ such that
the graph of $y=\xi_{v_k,p,j}(x)$ over $(\gamma_p,\gamma_{p+1})$
separates $D_{p,k}$ from $D_{p,k+1}$. We may now take
\begin{equation*}
d_p=g_{\mu(p,0)}+\sum_{k=0}^{s(p)-1}c(v_{p,k},p,j(p,k)),
\end{equation*}
by \eqref{Mahe's lemma} and \eqref{closed}.
\end{proof}

\begin{remark}
The above proposition proves the one-variable analog of
Theorem~\ref{PBCSIG}. For if the given function $h$ does not
involve one of the two variables (say, $x$), then by
Remark~\ref{no x involved} above, none of the functions
that we constructed in the sets $\mathcal A$ and $\mathcal B$
will involve $x$, either, whence we would be able to take
$L=0$ (which would mean that $H_0$ equals all of $\mathbb R_{++}^2$)
in \eqref{projection}--\eqref{corollary}, \eqref{Mahe's lemma},
and \eqref{one variable} above.
\end{remark}

\section{Conclusion of the proof of Theorem~\ref{PBCSIG}}
\label{proof}

Recall, after \eqref{h} we defined ${\mathcal A}=
\{\,g_i-g_j\mid i<j\,\}$, and we defined $\mathcal B$ to be the
set obtained from $\mathcal A$ by closing under the
operations~\eqref{closed under partial} and \eqref{r} with
respect to $y$. We got an $L\ge0$ and certain
$\gamma_p$ on the $x$-axis such that
$0=\gamma_0<\gamma_1<\cdots<\gamma_L<\gamma_{L+1}=\infty$,
and for each $p\in\{0,1,,\ldots,L\}$ we got \eqref{one variable} a
function $d_p(x,y):\mathbb R_{++}^2\to\mathbb R$ that
(1)~is a supremum of infima of finitely many generalized polynomial
functions and
(2)~agrees with $h$ on $H_p$ ($=(\gamma_p,\gamma_{p+1})\times
\mathbb R_{++}$).

Now let $\mathcal C$ be the subset of $\mathbb R[\mathbb R^2]$
obtained from ${\mathcal B}\cup\{\,x-\gamma_p\mid1\le p\le L\,\}$
by closing under the ``$x$-analogs''
of the operations \eqref{closed under partial} and \eqref{r};
i.e., interchanging $x$ and $y$ in \eqref{K}, \eqref{closed
under partial}, and \eqref{r}. Then we immediately obtain,
first, the following $x$-analog of Lemma~\ref{projection}
and its Corollary \ref{corollary}:

\begin{lemma}
There exist $M\in\mathbb N$ and
$\eta_1<\eta_2<\cdots<\eta_M\in\mathbb R_{++}$ such that,
writing $\eta_0=0$ and $\eta_{M+1}=\infty$,
and fixing any $q\in\{0,1,\ldots,M\}$,
the zeros, in the $q$th horizontal half-strip
$I_q:=\mathbb R_{++}\times(\eta_q,\eta_{q+1})$,
of all the $a\in\mathcal C$, are the graphs of continuous,
monotonic,$^{\ref{monotonic}}$
generalized semialgebraic functions $x=\zeta^{q,k}(y)$,
$k=1,2,\ldots,t(q)$ $($for a suitable $t(q)\in\mathbb N)$.
Moreover, for each $y\in(\eta_q,\eta_{q+1})$,
\begin{equation}
0=:\zeta^{q,0}(y)<\zeta^{q,1}(y)<\cdots<\zeta^{q,t(q)}(y)<
\zeta^{q,t(q)+1}(y):=\infty.
\label{x-increasing order}
\end{equation}
Consequently, the sets
\begin{equation*}
E_{q,k}:=\{\,(x,y)\mid
\eta_q<y<\eta_{q+1},\>\zeta^{q,k}(y)<x<\zeta^{q,k+1}(y)\,\},
\end{equation*}
for $k\in\{0,1,\ldots,t(q)\}$, are nonempty, pairwise-disjoint,
generalized semialgebraic cells
$($in particular, they are open and $($pathwise$)$ connected$)$,
and their union is a dense open subset of $I_q$. Moreover,
the $E_{q,k}$ are ``stacked'' one to the right of the other
in the $x$-direction,
so that for any $y\in(\eta_q,\eta_{q+1})$ and for any
$(t(q)+1)$-tuple $x_0,x_1,\ldots,x_{t(q)}\in\mathbb R_{++}$ for which
each $(x_k,y)\in E_{q,k}$, $x_0<x_1<\cdots<x_{t(q)}$.
Finally, for each $k$, there is a $p\in\{0,1,\ldots,L\}$
such that $E_{q,k}\subseteq H_p$ $($since the functions
$x-\gamma_1,\ldots,x-\gamma_L$ belong to $\mathcal C)$.\qed
\label{x-lemma}
\end{lemma}

The second immediate consequence of our choice of $\mathcal C$
is the following $x$-analog of Proposition~\ref{one variable}:

\begin{proposition}
\label{the other variable}
Let $h$, $\mathcal A$, $\mathcal C$, $M$,
$\eta_0,\eta_1,\ldots,\eta_{M+1}$,
$q$, and $I_q$ be as above.
There is a function $e_q:\mathbb
R_{++}^2\to\mathbb R$ that $(1)$~is a supremum of infima of
finitely many generalized polynomial functions $\in\mathbb
R[\mathbb R^2]$ and $(2)$~coincides
with $h(x,y)$ on $I_q$.\qed
\end{proposition}

Let
\begin{equation*}
Q=\{\,(q,k)\mid q\in\{0,1,\ldots,M\},\
k\in\{0,1,\ldots,t(q)\}\,\},
\end{equation*}
where $M$ and $t(q)$ are as in \eqref{x-lemma}.
Then
\begin{equation}
\bigcup_{(q,k)\in Q}E_{q,k}\text{
is a dense open subset of }\mathbb R_{++}^2,
\label{dense}
\end{equation}
by \eqref{x-lemma}.

\begin{lemma}
\label{finer partition}
There is a function $\nu:Q\to\{1,\ldots,l\}$ such that
$\forall(q,k)\in Q$, $E_{q,k}\subseteq A_{\nu(q,k)}^\circ$
$($in particular, $h=g_{\nu(q,k)}$ on $E_{q,k})$.
\end{lemma}

\begin{proof}
This follows from Lemma~\ref{union}(5) and Lemma~\ref{x-lemma}.
\end{proof}

\begin{remark}[on Definition~\ref{pgp}]
\label{aside}
We can now substantiate the statement in Remark~\ref
{announcement of Remark "aside"} above, viz.,
that in the definition of ``piecewise generalized polynomial function''
\eqref{pgp}, it was not necessary to require each $A_i$ to be a
generalized semialgebraic set in the case where $h$ is continuous,
since in that case we may (by \eqref{finer partition} and
\eqref{union}(3)) take each $A_i$ to be the closure of the
union of certain $E_{q,k}$, which is automatically generalized
semialgebraic.
\end{remark}

\begin{notation}
\label{Delta}
For $a,b\in\mathbb R\cup\{\pm\infty\}$ with $a<b$, let
\begin{equation*}
\Delta(a,b)=\{\,(x,y)\in\mathbb R^2\mid xy>0\ \&\ a<x+y<b\,\}.
\end{equation*}
(See Figure~\ref{DeltaFigure}.)
\end{notation}

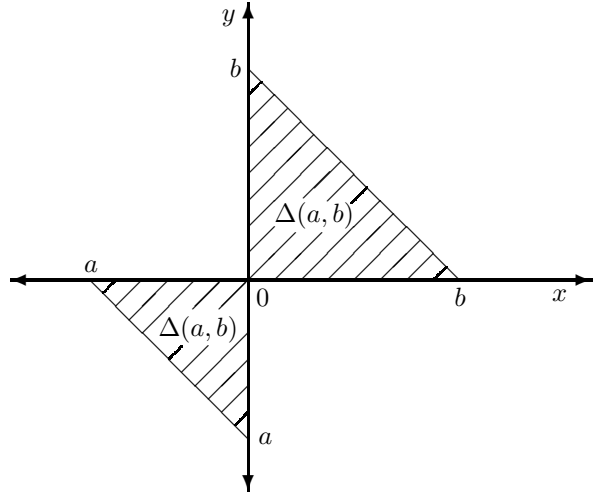
\begin{figure}[htb]
\setlength\parindent{0pt}
\begin{picture}(327,180)(-143.5,-0)% !!! This for normal size
%\begin{picture}(327,190)(-200,-10)% !!! This for large size.
{\thicklines
\put(-90, 80){\vector( 1,0){220}}
\put(130, 80){\vector(-1,0){220}}
\put(   0,-0){\vector(0, 1){185}}
\put(   0,185){\vector(0,-1){185}}
}
\put(3,70)0
\put(-10,178){$y$}
\put(115,72){$x$}
\put(-60,80){\line(1,-1){60}}
\put(10,102){$\Delta(a,b)$}
\put(0,160){\line(1,-1){80}}
\put(-34, 59){$\Delta(a,b)$}
\put( 78, 70){$b$}
\put(-62, 83){$a$}
\put( -7,157){$b$}
\put(  4, 18){$a$}
\qbezier(0,150,)(0,150)(5,155)
\put( 0,140){\line(1,1){10}}
\put( 0,130){\line(1,1){15}}
\put( 0,120){\line(1,1){20}}
\put( 0,110){\line(1,1){25}}
\put( 0,100){\line(1,1){30}}
\put( 0, 90){\line(1,1){10}}
\put(21,111){\line(1,1){14}}
\put( 0, 80){\line(1,1){18}}
\put(30,110){\line(1,1){10}}
\put(10, 80){\line(1,1){18}}
\qbezier(39,109)(39,109)(45,115)
\put(20, 80){\line(1,1){30}}
\put(30, 80){\line(1,1){25}}
\put(40, 80){\line(1,1){20}}
\put(50, 80){\line(1,1){15}}
\put(60, 80){\line(1,1){10}}
\qbezier(70,80)(70,80)(75,85)
\put(  0,80){\line(-1,-1){11}}
\qbezier(-25,55)(-25,55)(-30,50)
\put(-10,80){\line(-1,-1){11}}
\put(-20,80){\line(-1,-1){20}}
\put(-30,80){\line(-1,-1){15}}
\put(-40,80){\line(-1,-1){10}}
\qbezier(-50,80)(-50,80)(-55,75)
\put(-15,55){\line(-1,-1){10}}
\put(  0,60){\line(-1,-1){20}}
\put(  0,50){\line(-1,-1){15}}
\put(  0,40){\line(-1,-1){10}}
\qbezier(0,30)(0,30)(-5,25)
\end{picture}
\caption{The ``double-triangular'' region $\Delta(a,b)$
\eqref{Delta}. In this figure, $a<0<b$.}
\label{DeltaFigure}
\end{figure}

\begin{lemma}
\label{analytic}
Let $f(x,y)$ be a real-valued function that is analytic on a
neighborhood of $(0,0)$ in $\mathbb R^2$.
Write $f_x$ and $f_y$ for $\partial f/\partial x$ and
$\partial f/\partial y$, respectively.
Suppose $f(0,0)=0$, $f_x(0,0)>0$, and $f_y(0,0)>0$.
Then there is an $\epsilon>0$ such that for all
$(x,y)\in\Delta(0,\epsilon)$, $f(x,y)>0$.
\end{lemma}

\begin{proof}
By the Weierstrass Preparation Theorem and the theory of
Puiseux series (see, e.g., \cite[Propositions~3.3 and
4.4, respectively]{Ruiz 1993}),
the germ at $(0,0)$ of the zero-set of $f$ consists
of finitely many curve germs $(\alpha_1(t),\beta_1(t))$,
$(\alpha_2(t),\beta_2(t))$, \dots, where for each $i$:
$\alpha_i$ and $\beta_i$ are analytic for $0\le t<\delta$
(some $\delta>0$); $\alpha_i(0)=\beta_i(0)=0$; and
\begin{align}
\label{t^m}
\begin{split}
\text{either }
\alpha_i(t)&=t^{m_i}\text{ and }\beta'_i(0)\ne0,\\
\text{or }\beta_i(t)&=t^{m_i}\text{ and }\alpha'_i(0)\ne0,
\end{split}
\end{align}
for some $m_i\in\{1,2,\ldots\}$.  By the chain rule,
\begin{equation}
\label{chain}
0=\frac d{dt}\,0=\frac d{dt}\,f(\alpha_i(t),\beta_i(t))\bigr|_0
=f_x(0,0)\alpha'_i(0)+f_x(0,0)\beta'_i(0).
\end{equation}
Now we see that we cannot have both
$\alpha'_i(0)\ge0$ and $\beta'_i(0)\ge0$,
for this, together with \eqref{t^m} and the hypothesis of the lemma,
would make the right hand side of \eqref{chain} positive.
Thus there is an $\epsilon>0$ such that for all
$(x,y)\in\Delta(\epsilon)$, $f(x,y)\ne0$.
Since $\Delta(\epsilon)$ is connected and $f$ is
continuous and nonzero there, $f$ has constant sign
(positive or negative) throughout $\Delta(\epsilon)$.
This sign must, in fact, be positive, since
$\frac d{dt}f(t,t)\bigr|_0=f_x(0,0)+f_y(0,0)>0$ and $f(0,0)=0$.
\end{proof}

{\it Conclusion of the proof of Theorem~\ref{PBCSIG}\/}.
As in \cite{Mahe 1984}, the idea now is to construct, for each
two ordered pairs $(q,k)$ and $(r,m)\in Q$, a function
$u_{(q,k),(r,m)}$ that is the supremum of infima of finitely many
generalized polynomial functions, and is such that
\begin{equation}
\label{u}
u_{(q,k),(r,m)}
\begin{cases}
\le g_{\nu(q,k)}&\text{on }E_{q,k}\text{ and}\\
\ge g_{\nu(r,m)}&\text{on }E_{r,m}.
\end{cases}
\end{equation}
Then we shall be done, since the function
\begin{equation*}
u_{(r,m)}:=
\inf\bigl(\{g_{\nu(r,m)}\}\cup
\{\,u_{(q,k),(r,m)}\mid(q,k)\in Q\,\}\bigr)
\end{equation*}
will satisfy
\begin{alignat*}2
u_{(r,m)}&  = g_{\nu(r,m)}&&\text{ on }E_{r,m}\text{, and,}\\
\text{ for each }(q,k)\in Q,\quad
u_{(r,m)}&\le g_{\nu(q,k)}&&\text{ on }E_{q,k};
\end{alignat*}
then $h=\sup_{(r,m)\in Q}u_{(r,m)}$ throughout
$\bigcup_{(q,k)\in Q}E_{q,k}$, and hence (by \eqref{dense}
and the continuity of $h$) throughout $\mathbb R_{++}^2$,
as required.

So suppose $(q,k)$ and $(r,m)\in Q$,
and let us prepare to construct a $u_{(q,k),(r,m)}$
satisfying \eqref{u}. If $E_{\nu(q,k)}$ and
$E_{\nu(r,m)}$ are both subsets of the same horizontal
half-strip $I_q$ \eqref{x-lemma},\footnote{This will occur
if and only if $q=r$.} or of the same vertical half-strip $H_p$
(for some $p\in\{1,2,\ldots,L\}$, using the last sentence of
\eqref{x-lemma}), then we may take $u_{(q,k),(r,m)}$
to be either $e_q$ or $d_p$, respectively,
by \eqref{the other variable} or \eqref{one variable}.

The case that makes the proof for two variables harder
than the proof for one variable is the case when $E_{\nu(q,k)}$
and $E_{\nu(r,m)}$ do {\it not\/} lie in a common
half-strip (either horizontal or vertical). We may
assume, without loss of generality, that $E_{\nu(q,k)}$
is below and to the left of $E_{\nu(r,m)}$
(i.e., that points in $E_{\nu(q,k)}$ have $x$- and
$y$-coordinates less than the $x$- and $y$-coordinates
of points in $E_{\nu(r,m)}$, respectively);
the other three possibilities could be handled similarly.

$E_{\nu(q,k)}$ lies in the horizontal half-strip
$I_q:=\mathbb R_{++}\times(\eta_q,\eta_{q+1})$,
and in a unique vertical half-strip $H_p:=
(\xi_p,\xi_{p+1})\times\mathbb R_{++}$,
for some $p$. $E_{\nu(r,m)}$ lies in exactly
one of the horizontal half-strips $I_{q+1},I_{q+2},\dots$,
and in exactly one of the vertical half-strips $H_{p+1},
H_{p+2},\,$\dots. (See Figure~\ref{E},
where, for simplicity, $E_{\nu(r,m)}$ is shown lying in
$I_{q+1}$ and $H_{p+1}$.)
\begin{figure}[htb]
\setlength\parindent{0pt}
\begin{picture}(327,190)(-143.5,-10)% !!! This for normal size
%\begin{picture}(327,190)(-200,-10)% !!! This for large size.
{\thicklines
\put(-110, 0){\vector(1,0){280}}
\put(-90,-15){\vector(0,1){200}}
}
\put(-98,-10)0
\put(-103,175){$y$}
\put(155,-12){$x$}
\put(-30,  0){\line(0,1){10}}
\put(-30, 20){\line(0,1){10}}
\put(-30, 40){\line(0,1){10}}
\put(-30, 60){\line(0,1){10}}
\put(-30, 80){\line(0,1){10}}
\put(-30,100){\line(0,1){10}}
\put(-30,120){\line(0,1){10}}
\put(-30,140){\line(0,1){10}}
\put(-30,160){\line(0,1){10}}
\put( 40,  0){\line(0,1){10}}
\put( 40, 20){\line(0,1){10}}
\put( 40, 40){\line(0,1){10}}
\put( 40, 60){\line(0,1){10}}
\put( 40, 80){\line(0,1){10}}
\put( 40,100){\line(0,1){10}}
\put( 40,120){\line(0,1){10}}
\put( 40,140){\line(0,1){10}}
\put( 40,160){\line(0,1){10}}
\put(-90,60){\line(1,0){10}}
\put(-70,60){\line(1,0){10}}
\put(-50,60){\line(1,0){10}}
\put(-30,60){\line(1,0){10}}
\put(-10,60){\line(1,0){10}}
\put( 10,60){\line(1,0){10}}
\put( 30,60){\line(1,0){10}}
\put( 50,60){\line(1,0){10}}
\put( 70,60){\line(1,0){10}}
\put( 90,60){\line(1,0){10}}
\put(110,60){\line(1,0){10}}
\put(130,60){\line(1,0){10}}
\put(150,60){\line(1,0){10}}
\put(-90,100){\line(1,0){10}}
\put(-70,100){\line(1,0){10}}
\put(-50,100){\line(1,0){10}}
\put(-30,100){\line(1,0){10}}
\put(-10,100){\line(1,0){10}}
\put( 10,100){\line(1,0){10}}
\put( 30,100){\line(1,0){10}}
\put( 50,100){\line(1,0){10}}
\put( 70,100){\line(1,0){10}}
\put( 90,100){\line(1,0){10}}
\put(110,100){\line(1,0){10}}
\put(130,100){\line(1,0){10}}
\put(150,100){\line(1,0){10}}
\put(40,100){\circle*3}
\qbezier(40,140)(70,110)(110,100)
\put(80,100){\circle*3}
\put(80,-3){\line(0,1)6}
\put(67,-10){$\xi_{p+1}+b^*$}
\put(40,140){\circle*3}
\put(-93,140){\line(1,0)6}
\put(-131,137){$\eta_{q+1}+b^*$}
\put( -8,49){$x=\zeta^{q,k}(y)$}
\put(-10,51){\line(-1,0){10}}
\put(-20,51){\vector(1,2){9.7}}
\qbezier(-28,100)(-25,87)(-20,80)
\qbezier(-20, 80)(-10,66)( 25,60)
\put(53,71){$\Delta(a^*,b^*)+(\xi_{p+1},\eta_{q+1})$}
\put(51,73){\line(-1,0){10}}
\put(41,73){\vector(-1,2){8.5}}
\put(41,73){\vector(1,4){8.5}}
\put(-88,112){$x=\zeta^{q,k+1}(y)$}
\put(-29,114){\line(1,0){11}}
\put(-18,114){\vector(1,-3){7.5}}
\qbezier(-24,100)(-10,87)(10,85)
\qbezier( 10, 85)( 30,83)(40,60)
\put(40,140){\line(1,-1){40}}
\qbezier(40,130)(40,130)(45,135)
\put(40,120){\line(1,1){10}}
\put(40,110){\line(1,1){15}}
\put(40,100){\line(1,1){20}}
\put(50,100){\line(1,1){15}}
\put(60,100){\line(1,1){10}}
\qbezier(70,100)(70,100)(75,105)
\put(6.1,100){\line(1,-1){33.9}}
\qbezier(20,100)(20,100)(13.05,93.05)
\put(30,100){\line(-1,-1){11.95}}
\put(40,100){\line(-1,-1){16.95}}
\qbezier(40,90)(40,90)(37,87)
\qbezier(33,83)(33,83)(28.05,78.05)
\qbezier(37,77)(37,77)(33.05,73.05)
\put(6.1,100){\circle*3}
\put(6.1,-3){\line(0,1)6}
\put(-11,-10){$\xi_{p+1}-a^*$}
\put(40,66.1){\circle*3}
\put(-93,66.1){\line(1,0)6}
\put(-132,64){$\eta_{q+1}-a^*$}
\put( -32,-10){$\xi_p$}
\put( 38,-10){$\xi_{p+1}$}
\put(-100, 56){$\eta_q$}
\put(-109, 97){$\eta_{q+1}$}
\put(5,15){$H_p$}
\put(-76,77){$I_q$}
\put(78,130){$E_{\nu(r,m)}$}
\put(-3,73){$E_{\nu(q,k)}$}
\end{picture}
\caption{The case where $E_{\nu(q,k)}$ and $E_{\nu(r,m)}$
do not lie in a common half-strip.
(In this illustration, $E_{\nu(r,m)}$ lies in $I_{q+1}$
and $H_{p+1}$).)}
\label{E}
\end{figure}
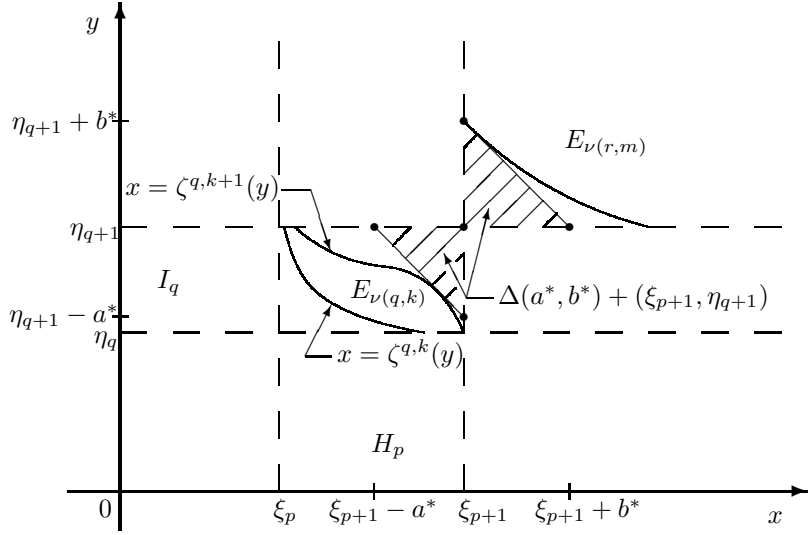

For any $a,b\in\mathbb R\cup\{\pm\infty\}$ with $a<b$,
write
\begin{equation*}
\Delta(a,b)+(\xi_{p+1},\eta_{q+1})=
\{\,(x+\xi_{p+1},\,y+\eta_{q+1})\mid(x,y)\in\Delta(a,b)\,\}.
\end{equation*}
Now let
\begin{align}
\label{a^*,b^*}
\begin{split}
a^*&=\min\{\,s\in\mathbb R\mid
(\Delta(s,0)+(\xi_{p+1},\eta_{q+1}))\cap E_{\nu(q,k)}
=\emptyset\,\}\text{ and}\\
b^*&=\max\{\,t\in\mathbb R\mid
(\Delta(0,t)+(\xi_{p+1},\eta_{q+1}))\cap E_{\nu(r,m)}
=\emptyset\,\}.
\end{split}
\end{align}
(Thus, $a^*\le0\le b^*$, by the assumptions on $E_{\nu(q,k)}$
and $E_{\nu(r,m)}$ made in the previous paragraph.)

To simplify notation, let
\begin{equation}
\label{g_nu(r,m)-g_nu(q,k)}
g(x,y)=g_{\nu(r,m)}(x,y)-g_{\nu(q,k)}(x,y).
\end{equation}
Pick any $e\in\mathbb N$ greater than every $x$- and
$y$-exponent ($\in\mathbb R$) occurring in (the unique
representation as in \eqref{signomial} of) $g(x,y)$.
There is a $T\ge a^*$ such that for all $(x,y)
\in\mathbb R_{++}^2$ with $x+y-\xi_{p+1}-\eta_{q+1}\ge T$,\footnote
{In particular, for all
$(x,y)\in\Delta(T,\infty)+(\xi_{p+1},\eta_{q+1})$.}
\begin{equation}
\label{(x+y-xi_{p+1}-eta_{q+1}-a^*)^e ge g(x,y)}
(x+y-\xi_{p+1}-\eta_{q+1}-a^*)^e\ge g(x,y).\footnote
{If we had allowed $e$ to be an arbitrary real number
(as opposed to an element of $e\in\mathbb N$),
then $(x+y-\xi_{p+1}-\eta_{q+1}-a^*)^e$ would not necessarily
be a signomial function (see \cite[Example~4.7]{Delzell 2008}).
Since, in fact, $e\in\mathbb N$, $(x+y-\xi_{p+1}-\eta_{q+1}-a^*)^e$
is a signomial function (it is even an ordinary polynomial).
We shall need this below.}
\end{equation}
We may assume that $T>b^*$ (in particular, $T>0$).

{\it Case 1\/}: $b^*-a^*>0$.
In this case, there is a $C\in\mathbb R$ such that
for all $(x,y)\in\Delta(b^*,T)+(\xi_{p+1},\eta_{q+1})$,
\begin{equation}
\label{b^*,T}
C\cdot(x+y-\xi_{p+1}-\eta_{q+1}-a^*)^e\ge g(x,y).\footnote
{Specifically, we may take
$C=(\max g(x,y))/\min((x+y-\xi_{p+1}-\eta_{q+1}-a^*)^e)$,
where the max and min are taken as $(x,y)$ ranges over the compact
set $\overline{\Delta(b^*,T)+(\xi_{p+1},\eta_{q+1})}$.
(Here we need $\min(x+y-\xi_{p+1}-\eta_{q+1}-a^*)>0$,
which follows from our assumption (here in case~1) that $b^*-a^*>0$.)}
\end{equation}
We may assume that $C\ge1$. Then we may take
\begin{equation*}
u_{(q,k),(r,m)}=
g_{q,k}(x,y)+C\cdot((x+y-\xi_{p+1}-\eta_{q+1}-a^*)^+)^e,
\end{equation*}
which satisfies \eqref{u} (using \eqref{a^*,b^*},
\eqref{(x+y-xi_{p+1}-eta_{q+1}-a^*)^e ge g(x,y)},
\eqref{b^*,T}, and \eqref{g_nu(r,m)-g_nu(q,k)}),
and which is a supremum of infima of finitely many
generalized polynomial functions (using Proposition~\ref
{closed}).

{\it Case 2\/}: $b^*-a^*=0$ (whence $a^*=0=b^*$).
In this case, let
\begin{equation*}
f(x,y)=g(x+\xi_{p+1},\,y+\eta_{q+1}).\footnote
{In general, $f$ is not a signomial function
(again, see \cite[Example~4.7]{Delzell 2008}),
but it is, at least, real analytic (for $x>-\xi_{p+1}$
and $y>-\eta_{q+1}$), and this is all we shall need.}
\end{equation*}
Pick any $D\in\mathbb R_{++}$ greater than
$\max\{f_x(0,0),\>f_y(0,0)\}$.
By Lemma~\ref{analytic}, there is an $\epsilon>0$ such that
$D\cdot(x+y)>f(x,y)$ for all $(x,y)\in\Delta(0,\epsilon)$;
equivalently,
\begin{equation}
\label{D}
D\cdot(x+y-\xi_{p+1}-\eta_{q+1})>g(x,y)
\end{equation}
for all $(x,y)\in\Delta(0,\epsilon)+(\xi_{p+1},\eta_{q+1})$.
We may assume that $\epsilon\le T$.

There is a $C\in\mathbb R$ such that
for all $(x,y)\in\Delta(\epsilon,T)+(\xi_{p+1},\eta_{q+1})$,
\begin{equation}
\label{epsilon,T}
C\cdot(x+y-\xi_{p+1}-\eta_{q+1})^e\ge g(x,y).\footnote
{Specifically, we may take
$C=(\max g(x,y))/\min((x+y-\xi_{p+1}-\eta_{q+1})^e)$,
where the max and min are taken as $(x,y)$ ranges over the compact
set $\overline{\Delta(\epsilon,T)+(\xi_{p+1},\eta_{q+1})}$.
(Here we need $\min(x+y-\xi_{p+1}-\eta_{q+1})>0$,
which follows from $\epsilon>0$.)}
\end{equation}
We may assume that $C\ge1$.

Then we may take
\begin{equation*}
u_{(q,k),(r,m)}=g_{q,k}(x,y)+\sup\{
D(x+y-\xi_{p+1}-\eta_{q+1})^+,
C((x+y-\xi_{p+1}-\eta_{q+1})^+)^e\},
\end{equation*}
which satisfies \eqref{u} (using \eqref{D},
\eqref{epsilon,T},
\eqref{(x+y-xi_{p+1}-eta_{q+1}-a^*)^e ge g(x,y)} (with $a^*=0$),
and \eqref{g_nu(r,m)-g_nu(q,k)}),
and which is a supremum of infima of finitely many
generalized polynomial functions (using Proposition~\ref
{closed}).
\qed

\end{document}